\documentclass[11pt]{amsart}
\usepackage{amsmath,amsthm,amssymb,amscd,cmbright}
\usepackage{amsrefs}
\usepackage[T1]{fontenc}
\usepackage[utf8]{inputenc}
\usepackage{lmodern}
\usepackage[pdftex]{hyperref}
\usepackage{amsfonts}
\usepackage{graphicx}
\usepackage{graphics}
\usepackage[english]{babel}
\usepackage{hhline}
\usepackage[dvips]{color}
\usepackage{pb-diagram}
\DeclareGraphicsExtensions{.png,.pdf,.jpg,.gif}
\usepackage{youngtab}

\newtheorem{teor}{Theorem}[section]

\theoremstyle{definition}
\newtheorem{rk}[teor]{Remark}
\newtheorem{example}[teor]{Example}

\newcommand{\R}{\mbox{${\mathbb R}$}}
\newcommand{\C}{\mbox{${\mathbb C}$}}

\newcommand{\Z}{\mbox{${\mathbb Z}$}}

\newcommand{\w}{\mbox{${\omega}$}}

\newcommand{\cpn}{\mbox{${\mathbb {CP}^n}$}}
\newcommand{\cpi}{\mbox{${\mathbb {CP}^1}$}}

\newcommand{\parcial}[2]{\frac{\partial#1}{\partial#2}}

\DeclareFontFamily{U}{mathx}{\hyphenchar\font45}
\DeclareFontShape{U}{mathx}{m}{n}{
      <5> <6> <7> <8> <9> <10>
      <10.95> <12> <14.4> <17.28> <20.74> <24.88>
      mathx10
      }{}
\DeclareSymbolFont{mathx}{U}{mathx}{m}{n}
\DeclareFontSubstitution{U}{mathx}{m}{n}
\DeclareMathAccent{\widecheck}{0}{mathx}{"71}
\DeclareMathAccent{\wideparen}{0}{mathx}{"75}

\begin{document}

\title{\bf Calabi Quasimorphisms for monotone coadjoint orbits}
\author{Alexander Caviedes Castro}
\email{alexanderc1@post.tau.ac.il}

\date{}

\begin{abstract}
We show the existence of Calabi quasimorphisms on the universal
covering $\widetilde{Ham}(\mathcal{O}_\lambda, \w_\lambda)$ of the
group of Hamiltonian diffeomorphisms of a monotone coadjoint orbit
$\mathcal{O}_\lambda$ of a compact Lie group. We show that this
result follows from positivity results of Gromov-Witten invariants
and the fact that the quantum product of Schubert classes can never
be zero.
\end{abstract}
\maketitle
\section{Introduction}

A quasimorphism of a group $G$ is a function $r:G\to \R$ for which
exists a constant $C\geq 0$ such that
$$
|r(g_1g_2)-r(g_1)-r(g_2)|\leq C \ \ \ \text{ for every }g_1, g_2\in
G.
$$
A quasimorphism $r$ is homogeneous if $r(g^n)=nr(g)$ for all $g\in
G$ and $n\in \Z.$


A result of Banyaga \cite{banyaga} states that the universal
covering $\widetilde{\operatorname{Ham}}(M, \w)$ of the group of
Hamiltonian diffeomorphisms of a closed symplectic manifold $(M,
\w)$ is perfect, i.e, it coincides with its commutator group. As a
consequence $\widetilde{\operatorname{Ham}}(M, \w)$ does not admit
non-trivial homomorphisms to $\R$ when $M$ is closed.


Even under these circumstances, we still may be able to construct
non-trivial quasimorphisms on $\widetilde{\operatorname{Ham}}(M,
\w)$ when $M$ is closed. Entov and Polterovich have constructed
non-trivial quasimorphisms on $\widetilde{\operatorname{Ham}}(M,
\w)$ that in addition satisfy the \textit{Calabi condition} with
techniques coming from Floer and quantum homology \cite{calabi},
\cite{intersection}, \cite{quasistates}, \cite{rigid}. More
precisely, Entov and Polterovich showed that if the quantum homology
algebra $QH_*(M)$ of a closed \textit{spherically monotone}
symplectic manifold $(M, \w)$ contains a field as a direct summand,
then $\widetilde{\operatorname{Ham}}(M, \w)$ admits a homogeneous
Calabi quasimorphism. In addition, Entov and Polterovich show that
these Calabi quasimorphisms have potential applications in several
areas of symplectic topology such as Hofer's geometry,
$C^0$-symplectic topology and Lagrangian intersection theory  (see
e.g. Polterovich and Rosen's book \cite{rosen}).

The quantum homology algebra $QH_*(M)$ contains a field as a direct
summand when $QH_*(M)$ is for instance \textit{semisimple}, i.e.,
$QH_*(M)$ decomposes into direct sum of fields. Examples of
symplectic manifolds with semi-simple quantum homology are
projective spaces (see e.g. Entov and Polterovich \cite{calabi}),
complex grassmannian manifolds, the smooth complex quadric
$\{z_0^2+z_1^2+\cdots + z_n^2=0\}\subset \cpn$ (see e.g. Abrams
\cite{quantumeuler}), toric Fano 2-folds (see e.g. Ostrover and
Tyomkin \cite{ostrover}, Entov and Polterovich \cite{quasistates}),
and generic symplectic toric manifolds (see e.g. Fukaya, Oh, Ohta,
Ohno \cite{fooo}, Ostrover and Tyomkin \cite{ostrover}).

The quantum homology algebra of a symplectic manifold is endowed
with a Frobenius structure. Abrams showed that a finite dimensional
Frobenius algebra is semisimple if and only if its \textit{Euler
class} is invertible. Abrams used this criterion to show that (any
specialization of) the quantum homology algebra of complex
grassmannian manifolds is semisimple \cite{quantumeuler}.

In this paper we focus our attention on coadjoint orbits
$\mathcal{O}_\lambda$ of compact Lie groups. Coadjoint orbits of any
Lie group are endowed with a symplectic form $\w_\lambda$ known as
the Kostant-Kirillov-Souriau form. The quantum cohomology algebra of
a coadjoint orbit is not always semisimple. Chaput, Manivel and
Perrin showed that the quantum homology algebra of the Grassmannian
manifold of isotropic 2-planes in the six dimensional complex space
$\C^6$ is not semisimple. They also showed that (any specialization
of) the quantum homology algebra of (co)minuscule homogeneous spaces
is semisimple \cite{exampleisotropic}.

Inspired by Abrams' result, in this paper we show that  a finite
dimensional Frobenius algebra over a field contains a field as a
direct sum if and only if its \textit{Euler class} is not nilpotent.
We use this criterion to show that the quantum homology algebra of a
coadjoint orbit of a compact Lie group always contains a field as a
direct sum, even when it is not necessarily semisimple, and thus
symplectic monotone coadjoint orbits of compact Lie groups admit
Calabi quasimorphisms.

We can express the quantum Euler class of the quantum algebra of a
coadjoint orbit $\mathcal{O}_\lambda$ of a compact Lie group in
terms of the Schubert basis $\{\sigma_u\}$ of its homology group
$H_*(\mathcal{O}_\lambda)$ as
$$
e_q=\sum_u \sigma_u*\check{\sigma}_{u},
$$
where in the last expression $*$ denotes the quantum product defined
on $QH_*(\mathcal{O}_\lambda)$ and $\{\check{\sigma}_{v}\}$ the
Schubert basis dual to $\{\sigma_u\}.$

The no nilpotency of the the quantum Euler class $e_q=\sum
\sigma_u*\check{\sigma}_{u}$ is a consequence of the following two
results: the first, the positiviy of genus zero Gromov-Witten
invariants whose constrains are Schubert classes (see e.g. Fulton
and Pandharipande \cite{fultonp}); and the second, the fact that the
quantum product of two Schubert classes does not vanish (Fulton and
Woodward \cite{fultonw}).

In the last part of this paper, we estimate from above the
Hofer-Zenhder capacity of a coadjoint orbit of a compact Lie group.
We use an inequality due to G. Lu \cite{Lu} where an upper bound for
the Hofer-Zehnder capacity of a symplectic manifold is estimated in
terms of its Gromov-Witten invariants. G. Lu's inequality \cite{Lu}
together with Fulton and Woodward's results \cite{fultonw} allow us
to express an upper bound for the Hofer-Zehnder capacity of a
coadjoint orbit in terms of its GKM-graph. We show that this
inequality is sharp for regular coadjoint orbits of the unitary
group $U(n).$

\section*{Acknowledgments}

I would like to thank Leonid Polterovich whose question gave the
impulse to write this paper. I also would like to thank Yaron
Ostrover for useful discussions. This research is supported by the
Israel Science Foundation grants $178/13$ and $1380/13.$

\section{Frobenius Algebras and Euler Class}

In this section we review briefly the defition of a Frobenius
algebra and its Euler class. We show that a Frobenius algebra
contains a field factor in its direct sum decomposition if and only
if its Euler class is not nilpotent. Most of the material presented
in this section is adapted from Abrams \cite{abrams},
\cite{quantumeuler}.

Let $K$ be a field of characteristic $0$ and let $(A,\, \cdot\,)$ be
a finite dimensional commutative algebra over $K$ with unity. A
\textbf{Frobenius structure} of $A$ is a linear form $f:A\to K$ such
that the bilinear form
\begin{align*}
\eta:A\times A &\to k \\
(a,b) &\mapsto f(a\cdot b)
\end{align*}
is nondegenerete. The pair $(A, f)$ is called a \textbf{Frobenius
algebra.}

Let $\{e_i\}$ be a basis for $A$ and $\{e_j^*\}$ be the
corresponding dual basis relative to $\eta,$ i.e., the basis such
that
$$
\eta(e_i, e_j^*)=f(e_i\cdot e_j^*)=\delta_{ij},
$$
for all $i$ and $j.$ The \textbf{Euler class} of  $(A, f)$ is the
element
$$
e_{A,f}:=\sum_i e_i\cdot e_i^*
$$
The Euler class $e_{A, f}$ is well defined and independent of the
choice of basis Abrams \cite[Proposition 5]{abrams}. Note that
$e_{A,f}$ is nonzero because
$$
f(e_{A, f}) = (A:K) \in K,
$$
and $K$ has characteristic 0.

Given two Frobenius algebras $(A, f)$ and $(B, g),$ their direct sum
is the pair $(A\oplus B, f\oplus g)$ where $A\oplus B$ denotes the
orthogonal direct sum of algebras whose product is defined
$$
(a_1, b_1)\cdot(a_2, b_2):=(a_1\cdot b_1, a_2\cdot b_2),
$$
and $f\oplus g$ acts by
\begin{align*}
f\oplus g:A\oplus B &\to K \\
a\oplus b &\mapsto f(a)+g(b)
\end{align*}
If $A$ is an arbitrary algebra with direct sum decomposition of
algebras $A=\bigoplus_i A_i,$ then $(A, f)$ is a Frobenius algebra
for some $f$ if and only if for every $i$ there is a $f_i$ such that
$(A_i, f_i)$ is a Frobenius algebra and $(A, f)=\bigoplus_i (A_i,
f_i)$ Abrams \cite{abrams}. Moreover, the Euler class $e_{A,f}$
respects direct sum decomposition of Frobenius algebras, i.e.,
$$
e_{A,f}=\bigoplus_i e_{A_i,f_i}
$$
We say that an algebra is \textbf{indecomposable} if it can not be
written as a direct sum of two nonzero submodules. The Krull-Schmidt
Theorem states that every finite dimensional algebra is a direct sum
of indecomposable algebras. This decomposition is unique up to
reordering of the summands (see e.g. Curtis and Reiner \cite[Theorem
14.5]{curtisreiner}).

For a Frobenius algebra $(A, f)$ we denote its ideal of nilpotents
by $\mathcal{N}(A).$ If the Frobenius algebra $(A,f)$ is
indecomposable then the ideal of nilpotents $\mathcal{N}(A)$
consists of all non-units of $A,$ and the annihilator of
$\mathcal{N}(A)$ is the principal ideal generated by the Euler class
$e_{A, f}$
$$
\operatorname{ann}(\mathcal{N}(A))=A\cdot e_{A,f}=\{a\cdot
e_{A,f}:a\in A\}
$$
(Abrams \cite[Proposition 3.3]{quantumeuler}). If $A$ is
indecomposable  and $\mathcal{N}(A)=0,$ then $A$ contains only units
and $0,$ so $A$ is just a field extension of $K,$ and in particular
the Euler class $e_{A,f}$ is a unit. On the other hand, if $A$ is
indecomposable and contains a nonzero nilpotent element $u,$ then
$u$ is a zero divisor of  $e_{A,f}$ and $e_{A,f}$ is not a unit, and
as a consequence the Euler class $e_{A,f}$ is nilpotent.

\begin{teor}\label{fieldfactor}
A  Frobenius algebra $(A,f)$ contains a field factor in its direct
sum decomposition if and only if its Euler class $e_{A,f}$ is not
nilpotent.
\end{teor}
\begin{proof}
If there exists a field extension $(\mathcal{F}, g)$ of $K$ such
that $(A, f)=(\mathcal{F}, g)\oplus (A', f')$ for some Frobenius
subalgebra $(A',f')\subset (A, f),$ then
$$
e_{(A, f)}=e_{(\mathcal{F}, g)}\oplus e_{(A', f')}
$$
The Euler class $e_{(\mathcal{F}, g)}$ is a unit and hence the Euler
class $e_{(A, f)}$ can not be nilpotent.

On the other hand, if we decompose the Frobenius algebra $(A,f)$
into a direct sum of indecomposable components
$$
(A,f)=\bigoplus_i (A_i, f_i)
$$
and none of them is a field, then the Euler classes $e_{(A_i, f_i)}$
of the indecomposable components $(A_i, f_i)$ are nilpotents, and
the Euler class
$$
e_{(A, f)}=\bigoplus_i e_{(A_i, f_i)}
$$
is nilpotent.
\end{proof}
\begin{teor}\label{square}
The characteristic element $e_{(A,f)}$ of a Frobenius algebra
$(A,f)$ is nilpotent if and only if $e_{(A,f)}^2=e_{(A,f)}\cdot
e_{(A,f)}=0.$ In particular, a Frobenius algebra $(A,f)$ contains a
field factor in its direct sum decomposition if and only if $e_{(A,
f)}^2\ne 0.$
\end{teor}
\begin{proof}
For any $\nu\in A$ there is a multiplication operator
\begin{align*}
M_\nu:A\to A\\
a \mapsto a\cdot \nu
\end{align*}
It is not hard to see that
$$
\operatorname{Trace}(M_\nu)=f(e_{(A, f)}\cdot
\nu)=\eta(e_{(A,f)},\nu)
$$
Now if the characteristic element $e_{(A,f)}$ is nilpotent, then for
any $\mu\in A,$ the element $\nu=e_{(A,f)}\cdot \mu $ is nilpotent,
and the corresponding multiplication operator
$M_{\nu}=M_{e_{(A,f)}\cdot\mu}$ on $A$ have vanishing trace, so that
$$
\operatorname{Trace}(M_\nu)=\eta(e_{(A,f)},\nu)=\eta(e_{(A,f)},
e_{(A,f)}\cdot \mu)=\eta(e_{(A,f)}^2, \mu)=0
$$
But $\eta$ is nondegenerate, and thus $e_{(A,f)}^2=0.$
\end{proof}
A Frobenius algebra $(A, f)$ is \textbf{semisimple} if it is a
direct sum of fields. The previous theorem is in contrast with
Abrams' criterion for semisimplicity of Frobenius algebras:

\begin{teor}[Abrams \cite{quantumeuler}]
A Frobenius algebra $(A, f)$ is semisimple if and only its Euler
class $e_{(A, f)}$ is a unit.
\end{teor}

\begin{example}
Let $f_1, \cdots, f_n \in \C[x_1, \cdots, x_n]$ be such that the
quotient ring
$$
A=\C[x_1, \cdots, x_n]/\langle f_1, \cdots, f_n \rangle,
$$
is finite dimensional. The Local Duality Theorem states that a
Frobenius structure for $A$ is given by the linear functional
\begin{align*}
f:A&\to \C \\
F \to \int_{|f_1|=\cdots=|f_n|=1} &\dfrac{F(x_1, \ldots,
x_n)}{\prod_{i=1}^nf_i(x_1, \ldots, x_n)}dx_1\ldots dx_n
\end{align*}
(see e.g. Griffiths and Harris \cite[Chapter 5]{GriffithsHarris}).
The Euler class of $(A, f)$ can be written as
$$
e_{A,f}=u\cdot
\det\Bigl(\parcial{f_i}{x_j}\Bigr)\operatorname{mod}{\langle f_1,
\cdots, f_n \rangle}
$$
for some unit $u$ of $A$ (Abrams \cite[Proposition 6.3]{abrams}).

Let $Z=\{\lambda\in \C^n: f_1(\lambda)=\ldots=f_n(\lambda)=0\}.$ The
Euler class $e_{(A, f)}$ is a unit in $A$ if and only if
$\det\Bigl(\parcial{f_i}{x_j}\Bigr)(\lambda)\ne 0$ for all $\lambda
\in Z.$ Likewise, the Euler class $e_{(A, f)}$ is nilpotent in $A$
if and only if $\det\Bigl(\parcial{f_i}{x_j}\Bigr)(\lambda)=0$ for
all $\lambda \in Z.$ So in conclusion,
\begin{enumerate}

\item The Frobenius algebra $(A, f)$ is semisimple if and only if the
multiplicity of any zero $\lambda\in Z$ is one.

\item The Frobenius algebra $(A, f)$ contains a field factor in its direct sum decomposition if and only if
there exists a zero $\lambda\in Z$ whose multiplicity is one.

\end{enumerate}

\end{example}



\section{Quantum Homology}

In this section we recall the definition of the quantum homology of
a monotone spherically symplectic manifold. A standard reference of
this material is McDuff and Salamon's book \cite{mcduff}.

Let $(M^{2n}, \w)$ be a closed symplectic manifold. The symplectic
manifold $(M^{2n}, \w)$ is spherically monotone if there exists a
real constant $\kappa>0$ such that for all $A\in \pi_2(M)$
$$
c_1(T_M)(A)=\kappa\w(A).
$$
Here $c_1(T_M)$ denotes the first Chern class of the bundle $(TM,
J),$ where $J$ is any almost complex structure compatible with $\w.$

Let $\bar{\pi}_2(M)=\pi_2(M)/\sim,$ where $A\sim B$ if and only if
$\w(A)=\w(B).$ If we assume that $\w$ does not vanish on $\pi_2(M),$
the group $\bar{\pi}_2(M)$ is an infinite cyclic group, and it has a
generator $S$ so $\w(S)>0.$ The integer number $N:=c_1(T_M)(S)$ is
called the \textbf{minimal Chern number} of the monotone symplectic
manifold $(M, \w).$

Let $A, B, C \in H_*(M, \Z)$  be three homology classes. Let $J$ be
a regular almost structure compatible with the symplectic form $\w.$
We denote by $\operatorname{GW}_k(A, B, C)$ the Gromov-Witten
invariant that, roughly speaking, counts the number of
$J$-holomorphic maps $\cpi \to M$ representing the class $kS\in
\bar{\pi}_2(M)$ and passing through generic representatives of the
homology classes $A, B$ and $C.$ The Gromov-Witten invariant
$\operatorname{GW}_k(A, B, C)$ is zero unless
$$
\deg{A}+\deg{B}+\deg{C}=4n-2kN
$$
Let $\Lambda$ be the field of formal Laurent series whose principal
part is a finite sum
\[
\Lambda:=\C[q]]=\left \{
  \begin{tabular}{c}
  $\sum_{k\in \mathbb{Z}}z_kq^k: z_k\in \C,$ and there exists \\$l\in \Z$
   such that  $z_k=0$ for all $k\leq l$
  \end{tabular}
\right \}
\]
As a vector space over $\C$ the \textbf{quantum homology algebra} of
$(M, \w)$ is defined to be
$$
QH_*(M):=H_*(M)\otimes_{\mathbb{C}}\Lambda
$$
The quantum multiplication $*$ on $QH_*(M)$ is defined as follows:
for $A, B \in H_*(M)$ and $k\in \Z_{\geq 1}$ define $(A*B)_k\in
H_*(M)$ as the unique class such that
$$
(A*B)_k\circ C=\operatorname{GW}_k(A, B, C)
$$
for all $C\in H_*(M).$ Here $\circ$ stands for the ordinary
intersection index in homology. Now for any $A, B \in H_*(M)$ set
$$
A*B=A\cap B +\sum_{k\in \mathbb{Z}_{\geq 1}}(A*B)_k\otimes q^k \in
QH_*(M)
$$
By $\Lambda$-linearity, we can extend the quantum product to the
whole $QH_*(M).$ We define a grading on $QH_*(M)$ so
$$
\deg{(A\otimes q^k)}:=\deg{A}-2kN
$$
for any $A\in H_*(M),$ and hence $\deg(a*b)=\deg(a)+\deg(b)-2n$ for
any $a, b\in QH_*(M).$

The quantum product defined on $QH_*(M)$ is skew-commutative,
associative and with unity equal to the fundamental class $[M\,].$
The even part
$QH_{\operatorname{ev}}(M):=H_{\operatorname{ev}}(M)\otimes_{\mathbb{C}}\Lambda$
is a commutative subalgebra of $QH_*(M).$

The algebra $QH_{\operatorname{ev}}(M)$ is a Frobenius algebra over
the Field $\Lambda$ with Frobenius algebra structure $f$ that
associates to a quantum homology class its coefficient at the
fundamental class of a point $[\operatorname{pt}]\in H_0(M)$
\begin{align*}
f: QH_{\operatorname{ev}}(M)&\to \Lambda\\
\sum_{A\in H_*(M)} A\otimes P_A(q)&\mapsto P_{[\operatorname{pt}]}
\end{align*}
The pairing
\begin{align*}
\eta:QH_{\operatorname{ev}}(M)&\times QH_{\operatorname{ev}}(M) \to
\Lambda \\
(a, b\,) &\mapsto f(a*b)
\end{align*}
is non-degenerate and $\eta(a*b,c)=\eta(a, b*c)$ for all $a,b,c\in
QH_{\operatorname{ev}}(M).$

\section{Calabi quasimorphisms and Quantum Homology}

In this section we remind the reader the relationship between
quantum homology and Calabi quasimorphisms. The material presented
in this section is mostly based on Entov and Polterovich's papers
\cite{calabi}, \cite{intersection}, \cite{quasistates},
\cite{rigid}.

Let $(M^{2n}, \w)$ be a closed connected symplectic manifold and
$I\subset \R$ be an interval containing $0.$ Given a smooth
Hamiltonian function $H:I\times M \to \R,$ set $H_t:=H(t, \cdot).$
We can associate the time dependent vector field $X_{H_t}$ defined
by
$$
\iota_{X_{H_t}}\w=dH_t
$$
We denote by $\phi_H^t$ the flow generated by $X_{H_t}.$
A Hamiltonian $H(x,t)$ of the closed symplectic manifold $(M, \w)$
is \textbf{normalized} if
$$
\int_M H_t\w=0  \text{ for all  } t\in I
$$
The group of \textbf{Hamiltonian diffeomorphisms} are the time maps
of Hamiltonian flows generated by normalized Hamiltonians
$$
\operatorname{Ham}(M, \w):=\{\phi_H^t: H \text{ (normalized)
Hamiltonian}\}
$$
Let $\widetilde{\operatorname{Ham}}(M, \w)$ be the universal
covering of $\operatorname{Ham}(M, \w)$
\begin{align*}
\widetilde{\operatorname{Ham}}(M, \w)&=\{(\phi, [\alpha]):\phi\in
\operatorname{Ham}(M, \w), \alpha=\{\alpha_t\}_{t\in [0,1]} \text{
is a smooth} \\ &\text{path of Hamiltonian diffeomorphisms with
}\alpha_0=I, \alpha_1=\phi\},
\end{align*}
here $[\alpha]$ stands for the homotopy class of $\alpha$ with fixed
endpoints.

For a non-empty open subset $U$ of $M,$ we denote by
$\widetilde{\operatorname{Ham}}_U(M, \w)$ the subgroup of
$\widetilde{\operatorname{Ham}}(M, \w)$ consisting of all elements
that can be represented by a path $\{\phi_H^t\}_{t\in [0,1]}$
starting at the identity and generated by a Hamiltonian function
$H_t$ supported in $U$ for all $t.$ Consider the map
\begin{align*}
\operatorname{Cal}_U&:\widetilde{\operatorname{Ham}}_U(M, \w) \to
\R\\
\phi&\mapsto \int_0^1\int_MH_t\w^n
\end{align*}
This map is well defined, i.e., it is independent of the choice of
the Hamiltonian functions generating $\phi.$ It is a group
homomorphism called the \textbf{Calabi homomorphism} (see e.g.
Polterovich and Rosen \cite[Chapter 4]{rosen}).

A non-empty subset $U$ of $M$ is called Hamiltonian displeceable if
there exists a Hamiltonian diffeomorphism $\phi\in
\operatorname{Ham}(M, \w)$ such that $\phi(U)\cap
\operatorname{Closure}(U)=\emptyset.$

A quasimorphism $\mu:\widetilde{\operatorname{Ham}}(M, \w)$ is
called a \textbf{Calabi quasimorphism} if it satisfies the following
two properties:

\begin{enumerate}

\item The map $\mu:\widetilde{\operatorname{Ham}}(M, \w)\to \R$ coincides with the Calabi homomorphism
$\operatorname{Cal}_U:\widetilde{\operatorname{Ham}}_U(M, \w)\to \R$
on any open and Hamiltonian displaceable set $U.$

\item For normalized Hamiltonians $F, G:M\times I\to \R$
$$
\int \min_M(F_t-G_t)dt \leq
\dfrac{\mu(\phi_G)-\mu(\phi_F)}{\operatorname{Vol}(M, \w)}\leq \int
\max_M(F_t-G_t)dt
$$

\end{enumerate}

Entov and Polterovich in \cite{calabi} have constructed
quasimorphisms on $\widetilde{\operatorname{Ham}}(M, \w)$ in terms
of \textbf{spectral invariants}. These invariants are given by a map
$$
c:QH_{\operatorname{ev}}(M)\times \widetilde{\operatorname{Ham}}(M,
\w) \to \R
$$
We refer the reader to Entov and Polterovich \cite{calabi}, Oh
\cite{Oh}, Polterovich and Rosen \cite{rosen}, Schwarz
\cite{schwarz}, Usher \cite{usher}, etc. for more details about
spectral invariants and their properties.

\begin{teor}[Entov and Polterovich \cite{calabi}]\label{entovpolterovich}
Let $(M, \w)$ be a closed monotone symplectic manifold. Assume that
the quantum homology algebra $QH_{\operatorname{ev}}(M)$ splits as
an algebra as $\mathcal{F}\oplus \mathcal{R},$ where $\mathcal{F}$
is a field. Let $e$ be the unit of $\mathcal{F}$ and
\begin{align*}
c_e:\widetilde{\operatorname{Ham}}(M, \w) &\to \R \\
\phi &\mapsto c(e, \phi)
\end{align*}
be the spectral invariant associated to $e.$ Then, the function
$c_e:\widetilde{\operatorname{Ham}}(M, \w) \to \R$ is a
quasimorphism, and the homogenization
$\mu:\widetilde{\operatorname{Ham}}(M, \w) \to \R$ of $c_e$ given by
$$
\mu(\phi)=\lim_{m\to \infty}\dfrac{c_e(\phi^m)}{m}
$$
is a homogeneous Calabi quasimorphism.
\end{teor}

\section{Geometry of coadjoint orbits}

In this section we recall some general statements about the geometry
of coadjoint orbits. Most of the material shown here can be found in
the classical literature such as Kirillov \cite{Orbit} for the
geometry of coadjoint orbits and Gelfand, Gelfand and Bernstein
\cite{GelfandBernstein} for the geometry of Schubert varieties.

Let $G$ be a compact Lie group, $\mathfrak{g}$ be its Lie algebra
and $\mathfrak{g}^*$ be the dual of $\mathfrak{g}$. Let $(\cdot\, ,
\cdot)$ denote an Ad-invariant inner product defined on
$\mathfrak{g}.$ We identify the Lie algebra $\mathfrak{g}$ and its
dual $\mathfrak{g}^*$ via this inner product. Let $\lambda\in
\mathfrak{g}^*$ and $\mathcal{O}_\lambda\subset \mathfrak{g}^*$ be
the coadjoint orbit passing through $\lambda.$ Let $\w_\lambda$ be
the \textbf{Kostant-Kirillov-Souriau form} defined on
$\mathcal{O}_\lambda$ by
$$
\w_{\lambda}(\hat{X}, \hat{Y})=\langle\lambda, [X, Y] \rangle \ \ \
X, Y \in \mathfrak{g},
$$
where $\hat{X}, \hat{Y}$ are the vector fields on $\mathfrak{g}^*$
generated by the coadjoint action of $G.$ The form $\w_\lambda$ is
closed and non-degenerate thus defining a symplectic structure on
$\mathcal{O}_\lambda.$

We denote by $G_{\mathbb{C}}$ the complexification of the Lie group
$G.$ Let $P\subset G_{\mathbb{C}}$ be a parabolic subgroup of
$G_{\mathbb{C}}$ such that $\mathcal{O}_\lambda \cong
G_{\mathbb{C}}/P.$ The quotient of complex Lie groups
$G_{\mathbb{C}}/P$ allows us to endow $\mathcal{O}_\lambda$ with a
complex structure $J$ compatible with $\w_\lambda$ so the triple
$(\mathcal{O}_\lambda, \w_\lambda, J)$ is a K\"ahler manifold. The
almost complex structure $J$ is regular in the sense of McDuff and
Salamon \cite[Proposition 7.4.3]{mcduff}.

Let $T\subset G$ be a maximal torus and let $B\subset
G_{\mathbb{C}}$ be a Borel subgroup with $T_{\mathbb{C}}\subset
B\subset P,$ where $T_{\mathbb{C}}$ denotes the complexification of
the maximal torus $T\subset G.$ Let $R \subset \mathfrak{t}^*$ be
the root system of $T$ in $G.$ Let $R^{+}\subset R$ be a system of
positive roots with simple roots $S \subset R^+.$ Let $W=N_{G}(T)/T$
be the Weyl group of $G.$ For every root $\alpha\in R,$ let
$s_\alpha \in W$ be the reflection associated to it. Recall that the
length $l(w)$ of $w\in W$ is defined as the minimum number of simple
reflections $s_\alpha\in W, \alpha \in S, $ whose product is $w.$
For the parabolic subgroup $P\subset G_{\mathbb{C}},$ let
$W_P=N_P(T)/T$ be the Weyl group of $P$ and $S_P\subset S$ be the
subset of simple roots whose corresponding reflections are in $W_P.$

Let $w_0$ be the longest element in $W$ and let
$B^{op}:=w_0Bw_0\subset G_{\mathbb{C}}$ be the \textbf{Borel
subgroup opposite} to $B.$ For $w\in W/W_P,$ let
$X(w):=\overline{BwP/P} \subset G_{\mathbb{C}}/P$ and
$Y(w):=\overline{B^{op}wP/P} \subset G_{\mathbb{C}}/P$ be the
\textbf{Schubert variety} and the \textbf{opposite Schubert variety}
associated with $w,$ respectively. We denote by $\sigma_w$ and
$\check{\sigma}_{w}$ the fundamental classes in the homology group
$H_*(G_{\mathbb{C}}/P, \Z)$ of $X(w)$ and $Y(w),$ respectively. Note
that $\check{\sigma}_{w}=\sigma_{w_ow}.$ For $w\in W/W_P,$ we let
$\check{w}:=w_0w\in W/W_P$ so
$\check{\sigma}_{w}=\sigma_{\check{w}}.$ The set of Schubert classes
$\{\sigma_w\}_{w\in W/W_P}$ forms a free $\Z$-basis of
$H_*(G_{\mathbb{C}}/P, \Z),$ and the set of Schubert classes
$\{\check{\sigma}_{w}\}_{w\in W/W_P}$ is the dual basis of
$\{\sigma_w\}_{w\in W/W_P}$ with respect to the intersection pairing
\begin{align*}
H_{*}(G_{\mathbb{C}}/P, \Z)&\otimes H_{*}(G_{\mathbb{C}}/P, \Z) \to \Z\\
(A, B)&\mapsto \int_{G_{\mathbb{C}}/P} A\cap B
\end{align*}
Now we explain when a coadjoint orbit $\mathcal{O}_\lambda$ is
symplectically monotone with respect to the Kostant-Kirillov-Souriau
form $\w_\lambda.$ Each root $\alpha \in R$ has a coroot
$\check{\alpha}\in \mathfrak{t}.$ The coroot $\check{\alpha}$ is
identified with $\frac{2\alpha}{(\alpha, \alpha)}\in \mathfrak{t}$
via the invariant inner product $(\cdot\,,\cdot).$ The system of
coroots is the set $\check{R}=\{\check{\alpha}:\alpha \in R\}$ and
the simple coroots is the set $\check{S}=\{\check{\alpha}:\alpha \in
S\}.$ For $\alpha\in R,$ let $\varpi_{\alpha}\in \mathfrak{t}^*$
denote the fundamental weight defined by
$$
\langle\varpi_\alpha, \check{\beta}\rangle=\delta_{\alpha, \beta},
$$
for any $\beta\in R.$ Here $\langle\cdot\, ,\cdot\, \rangle$ denotes
the standard pairing $\langle\cdot\, ,
\cdot\,\rangle:\mathfrak{t}\otimes \mathfrak{t}^* \to \R.$

The cohomology group $H^2(G_{\mathbb{C}}/P; \Z)$ can be identified
with the span
$$
\Z\{\varpi_\alpha:\alpha\in S\backslash S_P\}
$$
and the homology group $H_2(G_{\mathbb{C}}/P; \Z)$ with the quotient
$$
\Z\check{S}/\Z\check{S}_P
$$
For each $\alpha \in S \backslash S_P,$ we identify the class
$\sigma_{s_\alpha}\in H_2(G_{\mathbb{C}}/P, \Z)$ with
$\check{\alpha}+\Z\check{S}_P\in \Z\check{S}/\Z\check{S}_P$ and we
identify its Poincar\'e dual
$\operatorname{PD}(\sigma_{s_\alpha})\in H^2(G_{\mathbb{C}}/P, \Z)$
$H^2(G_{\mathbb{C}}/P, \Z)$ with $\varpi_\beta.$ When $\lambda \in
\mathfrak{t}^*,$ the cohomology class of the
Kostant-Kirillov-Souriau form $[\w_\lambda] \in
H^2(G_{\mathbb{C}}/P, \Z)$ is identified with $\lambda \in
\R\{\varpi_\alpha:\alpha\in S\backslash S_P\}\subset
\R\{\varpi_\alpha:\alpha\in S\}=\mathfrak{t}^*.$ The Poincar\'e
pairing $H_2(G_{\mathbb{C}}/P, \Z)\otimes H^2(G_{\mathbb{C}}/P,
\Z)\to \Z$ is compatible with the standard pairing $\langle\cdot\, ,
\cdot\,\rangle:\mathfrak{t}\otimes \mathfrak{t}^* \to \R.$ For two
degrees $c, d\in H_2(G_{\mathbb{C}}/P, \Z),$ we say that $c\leq d,$
if $c-d =\sum_{\alpha\in S\backslash S_P}b_\alpha\sigma_{s_\alpha}$
and $b_\alpha \in \Z_{\geq 0}$ for all $\alpha\in S\backslash S_P.$

Under these identifications,
$$
c_1(T_{G_{\mathbb{C}}/P})=\sum_{\gamma \in R^{+}\backslash
R^+_P}\gamma\in \Z\{\varpi_\alpha:\alpha\in S\backslash S_P\} \cong
H^2(G_{\mathbb{C}}/P, \Z)
$$
A coadjoint orbit $(\mathcal{O}_\lambda, \w_\lambda)\cong
(G_{\mathbb{C}}/P, \w_\lambda)$ is symplectically monotone if there
exists a real constant $\kappa>0$ such that for all
$\check{\alpha}\in \check{S}\backslash \check{S}_P$
$$
n_{\alpha}:=\int_{\sigma_{s_\alpha}} c_1(T_{G_{\mathbb{C}}/P})=
\Big\langle \sum_{\gamma \in R^{+}\backslash R^+_P}\gamma\, ,
\check{\alpha} \,\, \Big\rangle=\kappa\langle \lambda,
\check{\alpha} \rangle=\kappa\w_\lambda(\sigma_{s_\alpha}),
$$
or simply,
$$
\lambda=\dfrac{1}{\kappa}\sum_{\gamma \in R^{+}\backslash
R^+_P}\gamma
$$
The minimal Chern number of the coadjoint orbit
$\mathcal{O}_\lambda$ is given by
$$
N:=\gcd_{\alpha\in S\backslash S_P}n_\alpha.
$$
\section{Quantum hohomology of $G_{\mathbb{C}}/P$}

Now we define the quantum homology of $\mathcal{O}_\lambda\cong
G_{\mathbb{C}}/P.$ For a degree $d=\sum_{\alpha\in S\backslash S_P}
d_\alpha \sigma_{s_\alpha}\in H_2(G_{\mathbb{C}}/P, \Z)$ and $u, v,
w\in W/W_P,$ let $N_{u, v}^w(d\,)$ be the Gromov-Witten invariant
$\operatorname{GW}_d(\sigma_u, \sigma_v, \check{\sigma}_{w})$ that
is equal to the number of morphisms $\mu:\cpi \to X$ of degree $d$
such that for three given distinct points $p_1, p_2, p_3\in \cpi$
and three general $g_1, g_2, g_3\in G_{\mathbb{C}},$ $\mu(p_1)$ is
in $g_1\cdot X(u),$ $\mu(p_2)$ is in $g_2\cdot X(v),$ and $\mu(p_3)$
is in $g_3\cdot Y(w)$, in particular the Gromov-Witten invariants
$N_{u,v}^w(d)$ are nonnegative integer numbers (see e.g. Fulton and
Pandharipande \cite[Lemma 14]{fultonp}).

Let $\Lambda=\C[q]]$ be the field of Laurent series with complex
coefficients and finite principal part. Recall that the quantum
homology ring $QH_*(G_{\mathbb{C}}/P)$ is as $\Lambda$-module the
tensor product
$$
H_*(G_{\mathbb{C}}/P, \Z)\otimes_{\mathbb{C}}\Lambda,
$$
so the Schuber classes $\{\sigma_w=\sigma_w\otimes 1\}_{w\in W/W_P}$
form a basis for $QH_*(G_{\mathbb{C}}/P)$ over $\Lambda.$ The
quantum product $*$ is defined by
$$
\sigma_u* \sigma_v=\sigma_u\cap \sigma_v+\sum_{k\in \mathbb{Z}_{\geq
1}} (\sigma_u* \sigma_v)_k q^k$$ where
$$
(\sigma_u* \sigma_v)_k=\sum_{w\in W/W_P}
\Bigl(\sum_{\sum_{\alpha}d_\alpha n_\alpha=kN}
N_{u,v}^w(d\,)\Bigr)\sigma_w
$$
Let
$$
f:QH_*(G_{\mathbb{C}}/P)=H_*(G_{\mathbb{C}}/P)\otimes \Lambda \to
\Lambda
$$
be the map defined by
$$
f\Bigl(\sum_{u\in W/W_P} \sigma_u\otimes P_u(q)\Bigr)=P_e(q)
$$
The bilinear map
\begin{align*}
\eta: QH_*(G_{\mathbb{C}}/P)&\otimes QH_*(G_{\mathbb{C}}/P)\to \Lambda\\
(u, v)&\mapsto f(u*v)
\end{align*}
defines a Frobenius algebra structure on $QH_*(G_{\mathbb{C}}/P)$
over the field $\Lambda.$
\begin{teor}
For $G_{\mathbb{C}}/P,$ let $\sigma_w\in H_*(G_{\mathbb{C}}/P, \Z)$
be the Schubert class associated with $w\in W/W_P.$ The quantum
Euler class of the Frobenius algebra $(QH_*(G_{\mathbb{C}}/P), f)$
is equal to
$$
e_{q}=\sum_{w\in W/W_P}\sigma_w*\check{\sigma}_{w}\in
QH_*(G_{\mathbb{C}}/P),
$$
where $\check{\sigma}_{w}\in H_*(G_{\mathbb{C}}/P, \Z)$ denotes the
Schubert class opposite to $\sigma_w.$
\end{teor}
\begin{proof}
The Schubert classes $\{\sigma_w\}_{w\in W/W_P}$ form a basis for
$QH_*(G_{\mathbb{C}}/P)$ over $\Lambda.$ We just have to check that
the dual basis of $\sigma_w$ with respect to the Frobenius pairing
$\eta$ is $\check{\sigma}_{w},$ or equivalently
$$
\eta(\sigma_u, \check{\sigma}_{v})=f(\sigma_u * \check{\sigma}_{v})
=\delta_{uv}
$$
for all $u, v\in W/W_P.$

The basis $\{\check{\sigma}_{u}\}$ is dual to the basis
$\{\sigma_{u}\}$ with respect to the intersection pairing defined on
$H_*(G/P, \Z),$ and as consequence
$$
f(\sigma_u\cap \check{\sigma}_{v})=\delta_{uv}
$$
On the other hand, for a nonzero degree $d,$ the Gromov-Witten
invariant
$$
\operatorname{GW}_d(\sigma_u, \check{\sigma}_{v},
\check{\sigma}_{e})=\operatorname{GW}_d(\sigma_u,
\check{\sigma}_{v}, [G/P])=N_{u, \check{v}}^{e}(d)
$$
is zero (see e.g. McDuff and Salamon \cite[Exercise 7.1.6]{mcduff}),
therefore
$$
\eta(\sigma_u, \check{\sigma}_{v})=f(\sigma_u
*\check{\sigma}_{v})=\sum_{\sum_{\beta}d_\beta n_\beta=kN
}q^kN_{u,\check{v}}^e(d)=f(\sigma_u\cap
\check{\sigma}_{v})=\delta_{uv},
$$
and we are done.
\end{proof}

\begin{example}[Quantum Euler Class of $G(2,4)$]
Let $G(2, 4)$ be the complex Grassmannian manifold of 2-planes in
$\C^4.$ Shubert cycles in $H_*(G(2, 4))$ are parametrized by
elements in the quotient set $S_4/(S_2\times S_2).$ Elements in the
quotient set $S_4/(S_2\times S_2)$ are in one to one correspondence
with the partitions contained in the rectangle of size $2\times 2,$
i.e., the partitions
$$
\{0, 1, (1,1), 2, (2,1), (2,2)\}.
$$
We will denote the corresponding Schubert cycles in the homology
ring $H_*(G(2, 4))$ by
$$
\{\sigma_0, \sigma_{1}, \sigma_{(1, 1)}, \sigma_{2}, \sigma_{(2,1)},
\sigma_{(2, 2)}\}.
$$
In our convention, for a partition $\lambda$
$$
\dim_{\mathbb{C}}(\sigma_\lambda)=\dim_{\mathbb{C}}G(2,
4)-|\lambda|,
$$
where $|\lambda|$ denotes the size of the partition $\lambda.$

As a $\Lambda$-module
$$
QH_*(G(2; 4))\cong \operatorname{span}_{\Lambda}\langle \sigma_0,
\sigma_{1}, \sigma_{(1, 1)}, \sigma_{2}, \sigma_{(2,1)}, \sigma_{(2,
2)} \rangle
$$
The special Schubert classes $\{\sigma_{0}, \sigma_1, \sigma_2\}$
generate $QH_*(G(2;4))$ as a ring. Giambelli's formula together with
Pieri's formula determine all the multiplication of Schubert classes
in the Grassmannian manifold (see e.g. Bertram \cite{bertram}). We
summarize the quantum multiplications of Schubert classes in the
following table
\begin{center}
\begin{tabular}{|c|c c c c c c|}\hline
* & 1 & $\sigma_{1}$ & $\sigma_{2}$ & $\sigma_{(1\,,1)}$  & $\sigma_{(2\,,1)}$ &
$\sigma_{(2\,,2)}$\\ \hline
1 & 1 & $\sigma_{1}$ & $\sigma_{2}$ & $\sigma_{(1\,,1)}$  & $\sigma_{(2\,,1)}$ & $\sigma_{(2\,,2)}$\\
$\sigma_{1}$ & $\sigma_{1}$ & $\sigma_{2}+\sigma_{(1\,,1)}$ & $\sigma_{(2\,,1)}$ & $\sigma_{(2\,,1)}$  & $\sigma_{(2\,,2)}+q$ & $q\sigma_{1}$\\
$\sigma_{2}$ & $\sigma_{2}$ & $\sigma_{(2\,,1)}$ & $\sigma_{(2\,,2)}$ & $q$  & $q\sigma_{1}$ & $q\sigma_{(1\,,1)}$\\
$\sigma_{(1\,,1)}$ & $\sigma_{(1\,,1)}$ & $\sigma_{(2\,,1)}$ & $q$ & $\sigma_{(2\,,2)}$  & $q\sigma_{1}$ & $q\sigma_{2}$\\
$\sigma_{(2\,,1)}$ & $\sigma_{(2\,,1)}$ & $\sigma_{(2\,,2)}+q$ & $q\sigma_{1}$ & $q\sigma_{1}$  & $q\sigma_{2}+q\sigma_{(1\,,1)}$ & $q\sigma_{(2\,,1)}$\\
$\sigma_{(2\,,2)}$ & $\sigma_{(2\,,2)}$ & $q\sigma_{1}$ &
$q\sigma_{(1\,,1)}$ & $q\sigma_{2}$ & $q\sigma_{(2\,,1)}$ & $q^2$\\
\hline
\end{tabular}
\end{center}
The Quantum Euler class $e_q$ of $QH_*(G(2, 4))$ is equal to
\begin{align*}
e_q&=2\sigma_{(2,2)}*1+2\sigma_{1}* \sigma_{(2,1)}+ \sigma_{(1,1)}*
\sigma_{(1,1)}+\sigma_{2}*\sigma_{2}\\&=6 \sigma_{(2,2)}+2q
\end{align*}
The Quantum Euler class $e_q$ has inverse $\dfrac{3}{16q^2}
\sigma_{(2\,,2)}-\dfrac{1}{16q},$ and, in particular, the quantum
cohomology ring $QH_*(G(2, 4))$ is semisimple. For arbitrary
Grassmannian manifolds $G(k, n)$, Abrams has shown in \cite{abrams}
that the quantum homology ring $QH_*(G(k, n))$ (after specialization
of the quantum parameters) is semisimple by proving that the quantum
Euler class of $QH_*(G(k, n))$ is invertible.
\end{example}
The following is an example due to Chaput, Manivel and Perrin
\cite{exampleisotropic} of a homogeneous space $G_{\mathbb{C}}/P$
whose quantum cohomology ring $QH^*(G_{\mathbb{C}}/P)$ is not
semisimple.

\begin{example}[Quantum Euler class of $IG(2, 6)$]
Let us endow $\C^6$ with a symplectic form $\Omega.$ We denote by
$Sp(6, \C)$ the symplectic group of invertible linear
transformations on $\C^6$ that preserve $\Omega.$

We denote by $IG(2, 6)$ the isotropic Grassmannian of isotropic
2-planes in $\C^6.$ The dimension of $IG(2, 6)$ is equal to 7.
Schubert cycles in $H_*(IG(2, 6))$ are indexed by 1-strict
partitions $\lambda$ (see e.g. Chaput, Manivel and Perrin
\cite{exampleisotropic}), and in $H_*(IG(2, 6), \Z)$ there are
explicitly 12 of these classes
$$
\{\sigma_0, \sigma_1, \sigma_2, \sigma_{1,1}, \sigma_3,
\sigma_{2,1}, \sigma_4, \sigma_{3,1}, \sigma_{4,1}, \sigma_{3,2},
\sigma_{4,2}, \sigma_{4,3}\},
$$
the degree of the Schubert class $\sigma_{\lambda}$ is equal to
$$
2(\dim_{\mathbb{C}}(IG(2, 6))-(\lambda_1+\lambda_2).
$$
In the figure below, we draw the Hasse diagram of the Schubert
cycles in $H_*(IG(2, 6), \Z)$ ordered by inclusion of the
corresponding Schubert varieties

\begin{center}
 \includegraphics{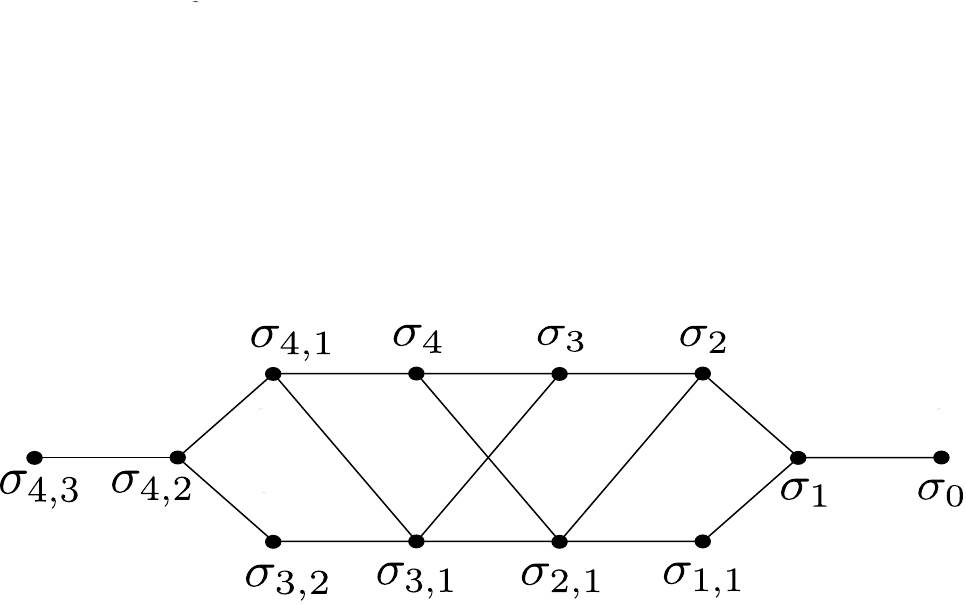}
\end{center}

The following is the multiplication table of the Schubert classes
$\sigma_\lambda$ with $\sigma_1, \sigma_2, \sigma_3$ in
$QH_*(IG(2,6))$ (see e.g. Chaput, Manivel and Perrin
\cite{exampleisotropic}, and Galkin, Mellin and Smirnov
\cite{smirnov})
\begin{center}
\begin{tabular}{|c|c c c c|}\hline
$*$ &  & $\sigma_1$ & $\sigma_2$ & $\sigma_3$ \\ \hline
$\sigma_1$ & $\sigma_1$ & $\sigma_2+\sigma_{1,1}$ & $\sigma_3+\sigma_{2,1}$  & $2\sigma_4+\sigma_{3,1}$ \\
$\sigma_2$ & $\sigma_2$  & $\sigma_3+\sigma_{2,1}$ & $2(\sigma_4+\sigma_{3,1})$& $2\sigma_{4,1}+\sigma_{3,2}+q\sigma_0$ \\
$\sigma_{1,1}$ & $\sigma^2-\sigma_2$ & $\sigma_{2,1}$ & $\sigma_4+\sigma_{3,1}$ & $\sigma_{4,1}+q\sigma_0$ \\
$\sigma_3$ & $\sigma_3$ & $2\sigma_4+\sigma_{3,1}$ & $2\sigma_{4,1}+\sigma_{3,2}+q\sigma_0$ & $2\sigma_{4,2}+q\sigma_1$\\
$\sigma_{2,1}$ &$\sigma_1*\sigma_2-\sigma_3$ & $\sigma_4+2\sigma_{3,1}$ & $2\sigma_{4,1}+\sigma_{3,2}+q\sigma_0$ & $\sigma_{4,2}+2\sigma_1q$\\
$\sigma_4$ & $\sigma_1*\sigma_{2,1}-2\sigma_{3,1}$ & $\sigma_{4,1}+\sigma_0q$ & $\sigma_{4,2}+q\sigma_1$ & $\sigma_{4,3}+q\sigma_2$\\
$\sigma_{3,1}$ & $-\frac{1}{3}\sigma_1*(\sigma_3-2\sigma_{2,1})$ & $\sigma_{4,1} +\sigma_{3,2}$ & $\sigma_{4,2}+q\sigma_1$ & $q(\sigma_2+\sigma_{1,1})$\\
$\sigma_{4,1}$& $\sigma_1*\sigma_4-q$ & $\sigma_{4,2}+q\sigma_1$ & $\sigma_{4,3}+q(\sigma_2+\sigma_{1,1})$ & $q(\sigma_{2,1}+\sigma_3)$\\
$\sigma_{3,2}$ &$\sigma_1*\sigma_{3,1}-\sigma_{4,1}$ & $\sigma_{4,2}$ & $q\sigma_2$ & $q\sigma_{2,1}$\\
$\sigma_{4,2}$ & $\sigma_1*\sigma_{4,1}-q\sigma_1$ & $\sigma_{4,3}+\sigma_2q$ & $q(\sigma_3+\sigma_{2,1})$ & $q(2\sigma_{3,1}+\sigma_4)$ \\
$\sigma_{4,3}$ & $\sigma_1*\sigma_{4,2}-q\sigma_2$ & $\sigma_3q$ &
$q(\sigma_4+\sigma_{3,1})$ & $q(\sigma_{4,1}+\sigma_{3,2})$ \\
\hline
\end{tabular}
\end{center}

The quantum Euler class of $QH_*(IG(2,6))$ equals to
\begin{align*}
e_q&=2(\sigma_0*\sigma_{4,3}+\sigma_1*\sigma_{4,2}+\sigma_2*\sigma_{4,1}+\sigma_{1,1}*\sigma_{3,2}+\sigma_3*\sigma_4+\sigma_{2,1}*\sigma_{3,1})\\
&=2(6\sigma_{4,3}+4\sigma_2q+\sigma_{1,1}q)
\end{align*}
The Euler class $e_q$ is not invertible because it has zero
divisors, for instance
\begin{align*}
e*\Bigl(\sigma_{4,3}-\sigma_2q+\sigma_{1,1}q\Bigr)=0.
\end{align*}
As a consequence, the quantum homology algebra $QH_*(IG(2,6))$ is
not semisimple.
\end{example}
In the next theorem, we show that the quantum homology algebra of
any coadjoint orbit of a compact Lie group contains a field factor
in its direct sum decomposition.
\begin{teor}\label{castro}
The quantum homology algebra $QH_*(G_{\mathbb{C}}/P)$ contains a
field factor in its direct sum decomposition.
\end{teor}
\begin{proof}
By Theorem \ref{fieldfactor}, the quantum cohomology algebra
$QH_*(G_{\mathbb{C}}/P)$ contains a field factor in its direct sum
decomposition if and only if the square of the quantum Euler class
$$
e_q=\sum_{w\in W/W_P}\sigma_w*\check{\sigma}_{w}
$$
is nonzero.
Note that the Euler class $e_q$ is of the form
$$
e_q=\dim_{\Lambda}{QH_*(G_{\mathbb{C}}/P)}[\operatorname{pt}]+\sum_{\substack{u\in
W/W_P \\ u\ne e}}\Bigl(\sum_{k \in \mathbb{Z}_{\geq
0}}q^kc_{k,u}\Bigr)\sigma_u,
$$
where $c_{k,u}$ are nonnegative integer numbers. Thus $e_q*e_q\ne
0,$ if for instance
$$
[\operatorname{pt}]*[\operatorname{pt}]\ne 0,
$$
or if there exists a Schubert cycle $\sigma_{u}$ and some degree
$c\in H_2(G_{\mathbb{C}}/P, \Z)$  so the Gromov-Witten invariant
$\operatorname{GW}_{c}([\operatorname{pt}], [\operatorname{pt}],
\sigma_{u})$ is non-zero.

Let $w_0$ be the longest element in $W.$ If we denote by $d$ the
degree of any $T$-invariant curve joining the coset $P/P$ with the
coset $w_0P/P$ in $G_{\mathbb{C}}/P,$ by the Theorem 9.1 of Fulton
and Woodward \cite{fultonw}, there exists a degree $c\leq d$ and a
Schubert class $\sigma_u$ so the Gromov-Witten invariant
$\operatorname{GW}_{c}([\operatorname{pt}], [\operatorname{pt}],
\sigma_{u})$ is nonzero, and we are done.
\end{proof}
As a consequence of the last result and Theorem
\ref{entovpolterovich}, we obtain the following theorem
\begin{teor}\label{quasi}
Let $(\mathcal{O}_\lambda, \w_\lambda)$ be a monotone coadjoint
orbit of a compact Lie group. Then there exists a homogeneous Calabi
quasimorphism on
$\widetilde{\operatorname{Ham}}(\mathcal{O}_\lambda, \w_\lambda).$
\end{teor}

\begin{rk}
In this paper we do not consider the question of whether the
quasimorphism given by Theorem \ref{quasi} descends to
$\operatorname{Ham}(\mathcal{O}_\lambda, \w).$  This is the same as
showing that the Calabi quasimorphism vanishes on
$\pi_1(\operatorname{Ham}(\mathcal{O}_\lambda, \w_\lambda)).$ The
Calabi quasimorphism restricted to
$\pi_1(\operatorname{Ham}(\mathcal{O}_\lambda, \w_\lambda))$ can be
written in terms of the Seidel representation (see e.g. Entov and
Polterovich \cite{calabi}, McDuff \cite{monodromy}), but it is a
difficult problem to compute the Seidel representation. However,
partial results have been obtained in this direction. Entov and
Polterovich showed that the Calabi quasimorpshims described in
Theorem \ref{quasi} descends when the coadjoint orbit
$\mathcal{O}_\lambda$ is isomorphic to a projective space
\cite{calabi}. Also, Branson showed that the Calabi quasimorphism
descends when the coadjoint orbit is isomorphic to the Grassmannian
manifold $G(2, 4)$ of 2-planes in $\C^4$ \cite{branson}. Branson
showed that the Seidel elements in $G(2, 4)$ have \textit{``finite
order''} and the quantum homology algebra $QH_*(G(2, 4))$ has
\textit{McDuff's ``property D''}; according to McDuff
\cite{monodromy}, when these two properties hold simultaneously, the
Calabi quasimorphism descends. Unfortunately, Branson also showed
that the projective spaces and the Grassmannian manifold $G(2, 4)$
are the only Grassmannian manifolds  satisfying property D. So even
the problem of deciding when the Calabi quasimorphism of a
Grassmannian manifold descends to
$\operatorname{Ham}(\mathcal{O}_\lambda, \w_\lambda)$ is still open.
\end{rk}

\section{Upper bound for the Hofer-Zehnder capacity}

In this section we point out that part of the argument given in the
proof of Theorem \ref{castro} can be used to estimate from above the
Hofer-Zenhder capacity of a coadjoint orbit of a compact Lie group.

Let $(M, \w)$ be a closed symplectic manifold. The oscillation of a
Hamiltonian function $H:M\to \R$ is
$$
\operatorname{osc}{H}:=\max{H}-\min{H}
$$
A Hamiltonian function $H:M\to \R $ is ($\pi_1$) admissible if all
(contractible) periodic orbits of the Hamiltonian vector field $X_H$
of period $< 1$ are constant.

The Hofer-Zehnder capacity of a closed symplectic manifold $(M, \w)$
may be expressed as
$$
\operatorname{c_{HZ}}(M, \w)=\sup{\bigl \{\operatorname{osc}{H}:
H:M\to\R \text{ admissible} \bigr\}},
$$
and the $\pi_1$-sensitive Hofer-Zehnder capacity of a closed
manifold $(M, \w)$ as
$$
\operatorname{c_{HZ}^\circ}(M, \w)=\sup{\bigl
\{\operatorname{osc}{H}: H:M\to\R \text{ $\pi_1$-admissible}
\bigr\}}
$$
The following theorem is due to G. Lu \cite{Lu}
\begin{teor}
Suppose that $(M, \w)$ admits a nonzero Gromov-Witten invariant of
the form
$$
\operatorname{GW}_{d}([\operatorname{pt}], [\operatorname{pt}],
A,\ldots, B)
$$
where $d\in H_2(M, \Z)/\text{torsion}$ and $A,\ldots, B$ are
rational homology classes of even degree. Then
$$
\operatorname{c_{HZ}^\circ}(M, \w)\leq \w(d\,),
$$
in particular
$$
\operatorname{c_{HZ}}(M, \w)\leq \w(d\,).
$$
\end{teor}

\begin{rk}
In \cite{usher}, Usher has provided a  Hamiltonian Floer Theory
proof of the last theorem. The tools that Usher developed in
\cite{usher} were also used by him to show a criterion for $(M, \w)$
to admit a Calabi quasimorphism. This criterion is among the same
lines as Entov and Polterovich's Theorem \ref{entovpolterovich}.
\end{rk}
For a coadjoint orbit $\mathcal{O}_\lambda \cong G_{\mathbb{C}}/P,$
let $w_0$ be the longest element in $W$. If we denote by $d$ the
degree of any $T$-invariant curve joining $P/P$ with $w_0P/P,$ by
the Theorem 9.1 of Fulton and Woodward \cite{fultonw}, there exists
a degree $c\leq d$ and a Schubert class $\sigma_u$ such that
$$
\operatorname{GW}_{c}([\operatorname{pt}], [\operatorname{pt}],
\sigma_{u})\ne 0.
$$
A $T$-invariant curve of a coadjoint orbit $\mathcal{O}_\lambda
\cong G_{\mathbb{C}}/P$ can be described in terms of its GKM-graph.
The GKM graph of $\mathcal{O}_\lambda$ is the graph whose vertices
are its $T$-fixed points and the edges are its irreducible
$T$-invariant curves. The collection of points $wP/P$ for $w\in W$
is the set of all $T$-fixed points in $G_{\mathbb{C}}/P.$ For each
positive root $\alpha \in R^+\backslash R^+_P$ there is a unique
irreducible $T$-invariant curve $C_\alpha$ that contains $P/P$ and
$s_\alpha P/P.$ Any other irreducible $T$-invariant curve is of the
form $w\cdot C_\alpha$ for some $w\in W.$  Any $T$-invariant curve
is a connected tree of irreducible $T$-invariant curves (see e.g.
Fulton and Woodward \cite{fultonw}).

All these remarks and its consequences regarding Hofer-Zehnder are
summarized in the following theorem
\begin{teor}
Let $\mathcal{O}_\lambda$ be a coadjoint orbit of a compact Lie
group $G,$ and let $\w_\lambda$ be the Kostant-Kirillov-Souriu form
defined on it. Then
$$
\operatorname{c_{HZ}}(\mathcal{O}_\lambda, \w_\lambda)\leq \min_A\,
\w_\lambda(A),
$$
where the minimum is taken over all degrees $A\in
H_2(\mathcal{O}_\lambda, \Z)$ of $T$-invariant curves joining $P/P$
with $w_0P/P$ in the GKM-graph.
\end{teor}
The following Example illustrates the previous theorem for regular
coadjoint orbits of $U(n).$
\begin{example}
Let $\lambda=(\lambda_1 > \lambda_2 > \ldots > \lambda_{n-1} >
\lambda_n)\in \R^n,$ and
$$
\mathcal{H}_\lambda:=\{A\in \mathfrak{u}(n): A^*=-A,\,
\operatorname{spectrum}{A}=-\lambda\}
$$
We can identify the set of skew-Hermitian matrices
$\mathcal{H}_\lambda$ with a regular coadjoint orbit of $U(n)$ via
the pairing
\begin{align*}
\mathfrak{u}(n)\times \mathfrak{u}(n) &\to \R \\ (X,
Y)&\mapsto\operatorname{Trace}(XY)
\end{align*}
Let $T=U(1)^n\subset U(n)$ be the maximal torus of diagonal matrices
in $U(n).$ The corresponding system of roots associated with the
torus $T$ is the set of vectors $\{e_i-e_j:i\ne j\}\subset
\mathfrak{t}\cong \mathfrak{t}^*.$

Any $T$-fixed point of $\mathcal{H}_\lambda$ with respect to the
conjugation action is a permutation of the diagonal matrix
$i(\lambda_1, \ldots, \lambda_n).$ Two $T$-fixed points of
$\mathcal{H}_\lambda$ are joined by one irreducible $T$-invariant
curve if they differ by one transposition.

In order to find an upper for the Hofer Zehnder capacity of
$\mathcal{H}_\lambda,$ we want to find a chain of irreducible
$T$-invariant curves joining the two $T$-fixed points $i(\lambda_1,
\lambda_2, \ldots, \lambda_n)$ and $i(\lambda_n, \lambda_{n-1},
\ldots, \lambda_1).$ Let us consider the following chain of
irreducible $T$-invariant curves
\begin{align*}
i(\lambda_1, \lambda_2, \ldots, \lambda_{n-1}, \lambda_n)
&\xrightarrow{(1,\, n)} i(\lambda_n, \lambda_2, \ldots,
\lambda_{n-1}, \lambda_1)\\ &\xrightarrow{(2,\, n-1)} \ldots \to
i(\lambda_{n}, \lambda_{n-1}, \ldots, \lambda_2, \lambda_1),
\end{align*}
The degree of this chain is equal to
$\sum_{k=1}^{[\frac{n-1}{2}]}[C_{(k\,, n-k)}]$ and its symplectic
area is equal to
$$
\w_{\lambda}\Bigl(\sum_{k=1}^{[\frac{n-1}{2}]}[C_{(k\,, n-k)}]
\Bigr)=\sum_{k=1}^{[\frac{n-1}{2}]}\langle\lambda,
\check{\alpha}_{k,\,n-k+1}\rangle=
\dfrac{1}{2}\sum_{k=1}^n|\lambda_k-\lambda_{n-k+1}|,
$$ where $\alpha_{k,\,n-k+1}$ denotes the root $e_k-e_{n-k+1},$
and thus
$$
\operatorname{c_{HZ}}(\mathcal{H}_\lambda, \w_\lambda)\leq
\dfrac{1}{2}\sum_{k=1}^n|\lambda_k-\lambda_{n-k+1}|
$$
Now we show that this inequality is sharp by constructing an
admissible Hamiltonian function $H:\mathcal{H}_\lambda \to \R$ whose
oscillation is equal to the right hand side of the last inequality.

The conjugation action of the torus $T$ on $\mathcal{H}_\lambda$ is
Hamiltonian with moment map given by
\begin{align*}
\mu: \mathcal{H}_\lambda &\to \R^n \\
A=(a_{ij}) &\to \operatorname{diagonal}(-iA)=(a_{11}, \ldots,
a_{nn})
\end{align*}
The image of the moment map $\mu$ is the convex hull of all possible
permutations of the vector $(\lambda_1, \ldots, \lambda_n)\in \R^n$
(see, e.g. Guillemin \cite{guillemin}).

For $t\in U(1)$ and $(m_1, \ldots, m_n)\in \Z^n,$ we will use the
convention that
$$
t^{(m_1, \ldots, m_n)}:=(t^{m_1}, \ldots, t^{m_n})\in
T=U(1)^n\subset U(n).
$$
Let
$$
\beta=\sum_{k=1}^{[\frac{n-1}{2}]}(e_k-e_{n-k+1})
$$
and $S=\{t^\beta=(t, t, \ldots, t^{-1}, t^{-1}):t\in S^1\}\subset
T.$ The action of the circle $S$ on $\mathcal{H}_\lambda$ is
Hamiltonian with moment map given by
\begin{align*}
\tilde{\mu}:\mathcal{H}_\lambda &\to \R \\
A=(a_{ij}) &\mapsto (\mu(A),
\beta)=a_{1,1}-a_{n,n}+a_{2,2}-a_{n-1,n-1}+\ldots
\end{align*}
The moment map image of $\tilde{\mu}$ is the interval
$$
\Bigl[-\dfrac{1}{2}\sum_{k=1}^n|\lambda_k-\lambda_{n-k+1}|,\dfrac{1}{2}\sum_{k=1}^n|\lambda_k-\lambda_{n-k+1}|\Bigr]\subset
\R,
$$
and thus the oscillation of $\tilde{\mu}$ is equal to
$\sum_{k=1}^n|\lambda_k-\lambda_{n-k+1}|.$ Unfortunately, the
function $\tilde{\mu}$ is not admissible. This is because, under the
action of $S$ on $\mathcal{H}_\lambda,$ there are elements in
$\mathcal{H}_\lambda$ with non-trivial finite stabilizers. All
possible stabilizer subgroups of $S$ are either $\{1\}, \Z_2$ or
$S.$ When the stabilizer subgroup of a skew-Hermitian matrix in
$\mathcal{H}_\lambda$ is $\Z_2,$ the period of the orbit passing
through the skew-Hermitian matrix is $1/2.$ Otherwise, the
skew-Hermitian matrix is either a $S$-fixed point or the period of
the orbit passing through the skew-Hermitian matrix is $1.$

The Hamiltonian function
$H=\dfrac{1}{2}\tilde{\mu}:\mathcal{H}_\lambda\to \R$ fixes this
problem. The orbits of $H$ are either constant or their periods are
either 1 or 2. So, $H$ is admissible, and
$$
\operatorname{osc}(H)=\dfrac{1}{2}\sum_{k=1}^n|\lambda_k-\lambda_{n-k+1}|
\leq c_{\operatorname{HZ}}(\mathcal{H}_\lambda, \w_\lambda)
$$
In conclusion,
$$
\operatorname{c_{HZ}}(\mathcal{H}_\lambda, \w_\lambda)=
\dfrac{1}{2}\sum_{k=1}^n|\lambda_k-\lambda_{n-k+1}|
$$
\end{example}

\renewcommand{\refname}{Bibliography}
\bibliographystyle{plain}
\bibliography{bibloquantumeuler}
\nocite{*}

\end{document}